\documentclass[reqno]{amsart}
\usepackage{graphicx}
\usepackage{amsmath}
\usepackage{mathrsfs}
\usepackage{amsfonts}
\usepackage{amssymb}
\usepackage{latexsym}
\usepackage{graphpap}

\newtheorem{theorem}{Theorem}[section]

\newtheorem{lemma}[theorem]{Lemma}
\newtheorem{proposition}[theorem]{Proposition}

\theoremstyle{definition}

\newcommand{\cont}{\operatorname{\mathsf{con}}}
\newcommand{\simp}{\operatorname{\mathsf{sim}}}

\theoremstyle{remark} \numberwithin{equation}{section}

\begin{document}

\title{A new example of limit variety of aperiodic monoids}

\thanks{This research was supported by the National
Natural Science Foundation of China (No.~10971086, 11371177)}

\author[W. T. Zhang]{Wen Ting Zhang$^\star$}\thanks{$^\star$Corresponding author}

\author[Y. F. Luo]{Yan Feng Luo}

\address{School of Mathematics and Statistics, Lanzhou University, Lanzhou, Gansu 730000, People's Republic of China; Key Laboratory of Applied Mathematics and Complex Systems, Gansu Province}
\email{zhangwt@lzu.edu.cn}
\email{luoyf@lzu.edu.cn}



\begin{abstract}
A limit variety is a variety that is minimal with respect to being non-finitely based. The two limit varieties of Marcel Jackson are the only
known examples of limit varieties of aperiodic monoids. Our previous work had shown that there exists a limit subvariety of aperiodic monoids that is different from Marcel Jackson's limit varieties. In this paper, we introduce a new limit variety of aperiodic monoids.
\end{abstract}

\maketitle

\section{Introduction} \label{sec intro}
A variety of algebras is \textit{finitely based} if there is a finite subset of its identities
from which all of its identities may be deduced, otherwise, the variety is
said to be \textit{non-finitely based}. An
algebra is \textit{finitely based} if it generates a finitely based
variety, otherwise, the algebra is said to be \textit{non-finitely based}. There are many finitely based and many non-finitely based
finite semigroups, and consequently the finite basis property for finite semigroups, and for finite algebras in general has
been one of the most extensively studied in facets of varieties. Refer to the surveys of Volkov \cite{Vol01} for a great deal of information on varieties, identities, and the finite basis problem for semigroups.

A variety is \textit{hereditarily finitely
based} if all its subvarieties are finitely based.  A variety is called \textit{limit variety} if
it is non-finitely based but every proper subvariety is finitely based; in other words,
limit varieties are precisely minimal non-finitely based varieties. Zorn's lemma
implies that each non-finitely based variety contains some limit subvariety; thus, a
variety is hereditarily finitely based if and only if it contains no limit subvarieties.
Therefore classifying hereditarily finitely based varieties in a certain sense reduces
to classifying limit varieties whence the latter task appears to be very hard in
general. Moreover, finding any concrete limit variety turns out to be nontrivial.
For example, no concrete limit variety of groups is known so far even though a
recent result by Kozhevnikov \cite{Kozhevnikov} shows that there are uncountably many of them.
Amongst (locally finite) groups, there are known to be infinitely many limit varieties, however the explicit construction of such an example remains one of the
foremost unsolved problems in group variety theory \cite{Kra03}
In contrast, for inverse semigroup varieties, a complete classification of non-group
limit varieties (and hence, a characterization of hereditarily finitely based varieties
modulo groups) has been found by Kleiman \cite{Kleiman}.

Recall that a monoid is \textit{aperiodic} if all its subgroups are
trivial. This article is concerned with the class $\mathfrak{A}$ of
aperiodic monoids and its subvarieties. A result of Kozhevnikov
implies the existence of continuum many limit varieties of monoids
consisting of groups~\cite{Kozhevnikov}. This makes classification of
limit varieties of monoids unfeasible unless restrictions are placed
on the groups lying in the variety.  The class $\mathfrak{A}$ is
arguably the most obvious natural candidate for attention. In the
early 2000s, Jackson proved that the variety $\mathsf{var} \{J_1\}$
generated by the monoid
\[
J_1 = \left\langle a,b,s,t \left|\, \text{$xy=0$ if $xy$ is not a factor of $asabtb$} \right. \right\rangle \cup \{1\} %
\]
of order 21 and the variety $\mathsf{var} \{J_2\}$ generated by the monoid
\[
J_2 = \left\langle a,b,s,t \left|\, \text{$xy=0$ if $xy$ is not a factor of either $absatb$ or $asbtab$} \right. \right\rangle \cup \{1\} %
\]
of order 35 are limit subvarieties of
$\mathfrak{A}$~\cite[Proposition~5.1]{Jac05fin}. As commented by
Jackson, no other similar examples of limit varieties could be found
\cite[Section~5]{Jac05fin}. This led him to pose the question~{\cite[Question~1]{Jac05fin}}:
Is there any finitely generated, non-finitely based
subvariety of $\mathfrak{A}$ that contains neither $\mathsf{var} \{J_1\}$
nor $\mathsf{var} \{J_2\}$?

In \cite{LeeZhang}, we show that the $\mathcal{J}$-trivial semigroup
\[
L=\left\langle\,a,b\left|\,a^2=a,\,b^2=b,\,aba=0\right.\right\rangle %
\]
of order six is one of minimal non-finitely semigroups. Let $L^1$ denote the monoid obtained by adjoining an identity element to $L$ and let $\mathsf{var} \{L^1\}$ denote the variety generated by $L^1$. It is easy to see that $\mathsf{var} \{L^1\}$ is a subvariety of $\mathfrak{A}$ that contains neither $\mathsf{var} \{J_1\}$ nor
$\mathsf{var} \{J_2\}$. In \cite{zhang13}, we show that $\mathsf{var} \{L^1\}$
is non-finitely based, and so there exists a limit subvariety of $\mathfrak{A}$ that is different from $\mathsf{var} \{J_1\}$ and $\mathsf{var} \{J_2\}.$ Consequently, identify all limit subvarieties of $\mathsf{var} \{L^1\}$ is an unavoidable step in the classification of limit varieties of aperiodic monoids.

The main goal of this paper and its prequel is to give an explicit example of a limit variety of $\mathfrak{A}$.
Let $A^1$ denote the monoid obtained by adjoining an identity element to the semigroup $A = \{ 0,a,b,c,d,e\}$ given by the following multiplication table: 
\[
\begin{array}
[c]{c|cccccc}
A & 0 & a & b & c & d & e \\ \hline %
0 & 0 & 0 & 0 & 0 & 0 & 0 \\ %
a & 0 & 0 & 0 & 0 & 0 & a \\ %
b & 0 & 0 & 0 & 0 & 0 & b \\ %
c & 0 & 0 & a & 0 & c & 0 \\ %
d & 0 & 0 & b & 0 & d & 0 \\ %
e & 0 & a & a & c & c & e
\end{array}
\]
The semigroup $A$ was first investigated by Lee and Zhang~\cite[Section~19]{LeeZhang}, where it was shown to be finitely based. Let $B^1$ be the semigroup that is dual to $A^1$. %
In \cite{zhang16}, by using a sufficient condition, we have shown that the semigroup $A^1 \times B^1$ is non-finitely based. In this paper, all of proper monoid subvarieties of the variety generated by $A^1 \times B^1$ are shown to be finitely based, and the monoid subvariety lattice of $\mathsf{var}_{\mathbb{M}} \{A^1 \times B^1\}$ will be completely described. Hence the monoid variety $\mathsf{var}_{\mathbb{M}} \{A^1 \times B^1\}$ is a limit monoid variety. Also,
an identity basis $A^1 \times B^1$ will be given, the finite membership problem for the variety generated by $A^1 \times B^1$ admits a polynomial algorithm.

In section \ref{sec basis}, an identity basis for $\mathsf{var} \{A^1 \times B^1\}$ will be given.
In section \ref{sec lattice}, all monoid subvarieties of $\mathsf{var} \{A^1 \times B^1\}$ will be characterized and each of them is finitely based. Furthermore, the monoid subvariety lattice of $\mathsf{var}_{\mathbb{M}} \{A^1 \times B^1\}$ will be completely described. Hence $\mathsf{var}_{\mathbb{M}} \{A^1 \times B^1\}$ is a limit monoid variety.

Recall that a variety is \textit{small} if it has finitely many subvarieties, a small variety
is \textit{cross} if it is finitely based and finitely generated, and a non-cross variety is
\textit{almost cross} if all its proper subvarieties are Cross.
Hence $\mathsf{var} \{A^1 \times B^1\})$ is a small almost cross variety.

\section{Preliminaries} \label{sec prelim}

Most of the notation and background material of this article are
given in this section. Refer to the monograph~\cite{BurSan81} for
any undefined terminology.

Let $\mathcal{X}$ be a fixed countably infinite alphabet
throughout. For any subset $\mathcal{A}$ of $\mathcal{X}$, denote
by $\mathcal{A}^{\ast}$ the free monoid over $\mathcal{A}$.
Elements of $\mathcal{X}$ and $\mathcal{X}^{\ast}$ are referred to
as \textit{letters} and \textit{words} respectively.

The \textit{content} of a
word $\mathbf{w}$, denoted by $\mathsf{con} (\mathbf{w})$, is the
set of letters occurring in $\mathbf{w}$; The \textit{multiplicity} of a letter $x$ in $\mathbf{w}$, denoted by
$\mathsf{m} (x, \mathbf{w})$, is the number of times $x$ occurs in
$\mathbf{w}$; A letter $x$ is
\textit{simple} in a word $\mathbf{w}$ if $\mathsf{m} (x,
\mathbf{w}) = 1$; otherwise, $x$ is \textit{non-simple} in
$\mathbf{w}$. The set of simple letters of a word $\mathbf{w}$ is
denoted by $\mathsf{sim} (\mathbf{w})$ and the set of non-simple
letters of $\mathbf{w}$ is denoted by $\mathsf{non} (\mathbf{w})$.
Note that $\mathsf{con} (\mathbf{w}) = \mathsf{sim} (\mathbf{w})
\cup \mathsf{non} (\mathbf{w})$ and $\mathsf{sim} (\mathbf{w}) \cap
\mathsf{non} (\mathbf{w}) = \emptyset$.

An identity is typically written as $\mathbf{w} \approx
\mathbf{w}'$ where $\mathbf{w}$ and $\mathbf{w}'$ are nonempty
words. Let $\Pi$ be any set of identities. The deducibility of an
identity $\mathbf{w} \approx \mathbf{w}'$ from $\Pi$ is indicated
by $\Pi \vdash \mathbf{w} \approx \mathbf{w}'$ or $\mathbf{w}
\stackrel{\Pi}{\approx} \mathbf{w}'$. The variety \textit{defined}
by $\Pi$, denoted by $\mathbb{V}(\Pi)$, is the class of all semigroups that satisfy all
identities in $\Pi$; in this case, $\Pi$ is said to be a
\textit{basis} for the variety.

For any word $\mathbf{w}$ and any set $\mathcal{B}$ of letters in
$\mathbf{w}$, let $\mathbf{w}_\mathcal{B}$ denote the word obtained
from $\mathbf{u}$ by retaining the letters from $\mathcal{B}$ (but
removing all others). It is easy to see that if the identity
$\mathbf{u} \approx \mathbf{v}$ is satisfied by a monoid $M$, then
any identity of the form $\mathbf{u}_\mathcal{B} \approx
\mathbf{v}_\mathcal{B}$ is also satisfied by $M$.

For any class $\mathfrak{C}$ of semigroups or monoids, let $\mathsf{var} \{\mathfrak{C}\}$ denote the semigroup variety generated by $\mathfrak{C}$. For any class $\mathfrak{C}$ of monoids, let $\mathsf{var}_{\mathbb{M}} \{\mathfrak{C}\}$ denote the monoid variety generated by $\mathfrak{C}$.
It is easy to see that a monoid $S$ is contained in the semigroup variety
of a monoid $H$ if and only if it is contained in the monoid variety of $H$.

The following small semigroups are required throughout the article:
\begin{align*}
J & =\Big\langle a,b\,\Big|\,ab=0,\,ba=a,\,b^2=b\Big\rangle, \\ %
A_0 & = \Big\langle \, a,b \,\Big|\, a^{2} = a, \, b^2 = b, \, ba = 0 \Big\rangle, \\%
B_0 & = \Big\langle \, a,b,c \,\Big|\, a^{2} = a, \, b^2= b, \, ab = ba =0, \, ac=cb=c \Big\rangle, \\%
L & = \Big\langle \, a,b \,\Big|\, a^{2} = ba = 0, \,  ab = a, \, b^2 = b \Big\rangle, \\%
R & = \Big\langle \, a,b \,\Big|\, a^{2} = ab = 0, \,  ba=a, \,  b^{2} = b\, \Big\rangle, \\%
M & = \Big\langle \, a,b,c \,\Big|\,  cb=a, \, \mbox{and all other products equal to $0$} \Big\rangle, \\%
N & = \Big\langle \, a \,\Big|\, a^{2} = 0 \Big\rangle. \\%
\end{align*}
For any non-unital semigroup $S$, let $S^1$ denote the monoid
obtained by adjoining a unit element to $S$.

\begin{proposition}
\begin{enumerate}
\item $\mathbf {A}_{0}^{1}= \mathsf{var}\{x^3\approx x^2, x^{2}yx\approx xyx \approx xyx^{2}, xyhxty \approx yxhxty, xhytxy \approx xhytyx\};$%
\item $\mathbf {B}_{0}^{1}= \mathsf{var}\{x^3\approx x^2, x^{2}yx\approx xyx \approx xyx^{2}, xyhxty \approx yxhxty, xhytxy \approx xhytyx, x^2y^2\approx y^2x^2\};$%
\item $\mathbf {L}^{1}= \mathsf{var}\{x^3\approx x^2, xyx\approx x^2y, x^2y^2 \approx y^2x^2 \};$%
\item $\mathbf {R}^{1}= \mathsf{var}\{x^3\approx x^2, xyx\approx yx^2, x^2y^2 \approx y^2x^2 \};$%
\item $\mathbf {M}^{1}= \mathsf{var}\{x^3\approx x^2, xyx\approx x^2y, xyx \approx yx^2 \};$%
\item $\mathbf {N}^{1}= \mathsf{var}\{x^3\approx x^2, xy\approx yx \};$%
\end{enumerate}
\end{proposition}

For any letters $x$ and $y$ of a word $\mathbf{w}$, write $x
\prec_\mathbf{w} y$ to indicate that within $\mathbf{w}$, each
occurrence of $x$ precedes every occurrence of $y$. In other words,
if $x \prec_\mathbf{w} y$ with $p=\mathsf{m}(x,\mathbf{w})$ and
$q=\mathsf{m}(y,\mathbf{w})$, then retaining only the letters $x$ and $y$
in $\mathbf{w}$ results in the word $x^p y^q$.

\begin{lemma}[{\cite[Lemma~1.3]{LeeZhang}}] \label{prelim LEM J1 words}
Suppose that $\mathbf{w} \approx \mathbf{w}'$ is any identity
satisfied by the semigroup $J^1.$ Then
\begin{enumerate}
\item $\cont (\mathbf{w}) = \cont (\mathbf{w}')$ and $\simp (\mathbf{w}) = \simp (\mathbf{w}');$ %

\item for any $x \in \cont (\mathbf{w}) = \cont (\mathbf{w}')$ and $y \in \simp (\mathbf{w}) = \simp (\mathbf{w}'),$ the conditions $x \prec_\mathbf{w} y$ and $x \prec_{\mathbf{w}'} y$ are equivalent\emph{;} %

\item $\mathbf{w}_{\simp} = \mathbf{w}_{\simp}'.$ %
\end{enumerate}
\end{lemma}

\section{An identity basis for $\mathsf{var} \{A^1 \times B^1\}$} \label{sec basis}

The present section establishes an identities basis for the variety  $\mathsf{var} \{A^1 \times B^1\}.$
\begin{theorem}
\label{S basis theorem} The variety $\mathsf{var} \{A^1 \times B^1\}$ is defined
by the identities
\begin{gather}
x^{2} \approx x^{3} , \ \  xyx \approx x^{2}yx \approx xyx^{2}, \label{basis S xx} \\ %
xy^{2}x \approx (xy)^2 \approx (xy)^2x \approx yx^{2}y \approx (yx)^2, \label{basis S xy2x} \\ %
xytxsy \approx (xy)^{2}txsy, \ \ xt(yx)^{2}sy \approx xtyxsy, \ \ xtysxy \approx xtys(xy)^{2}, \label{basis S txsy} \\%
xy_{1}^{2} y_{2}^{2} \cdots y_{n}^{2} x \approx x y_{1}^{2} x y_{2}^{2} x \cdots y_{n}^{2} x, \qquad n=2,3,\ldots, \label{basis S xy2z2x} \\%
xytxz_{1}^{2} \cdots z_{n}^{2} y \approx yxtx z_{1}^{2} \cdots z_{n}^{2} y, \qquad n=0,1,\ldots, \label{basis S txaay} \\%
xz_{1}^{2} \cdots z_{n}^{2} ytxy \approx xz_{1}^{2} \cdots z_{n}^{2} y tyx,  \qquad n=0,1,\ldots. \label{basis S xaayt} %
\end{gather}
\end{theorem}

\noindent The proof of Theorem~\ref{S basis theorem} is given at
the end of the section.

Most of the equational deductions in this article are deductions
within the equational theory of $\mathsf{var} \{A^1 \times B^1\}$. Therefore, it
will be convenient to refer to the identities in Theorem~\ref{S
basis theorem} collectively by $\circledS$, that is,
\[
\circledS = \{ (\ref{basis S xx}), (\ref{basis S xy2x}), (\ref{basis S txsy}), (\ref{basis S xy2z2x}),
(\ref{basis S txaay}), (\ref{basis S xaayt}) \}.
\]
For any sets $\Pi_{1}$ and $\Pi_{2}$ of identities, the deduction
$\circledS \cup \Pi_{1} \vdash \Pi_{2}$ is abbreviated to $\Pi_{1}
\Vdash \Pi_{2}$.

For any nonempty set $Z = \{ z_1,\ldots,z_r \}$ of letters, the word of the form
\[
(z_1 \cdots z_r)^2, %
\]
is said to be the \textit{$Z$-square},
in particular, if $z_1, \ldots, z_r$ are in alphabetical order, then it said to be the
\textit{perfect $Z$-square}.
More generally, a \textit{square} (resp. \textit{perfect square}) is a $Z$-square (resp. perfect $Z$-square)
for some nonempty set $Z$ of letters.

\begin{lemma}\label{lem: perfect square}
Let $\mathbf{z}$ be any square. Then the identities~$\circledS$ imply the identity
\begin{equation}
\mathbf{z} \approx \overline{\mathbf{z}}, \label{perfect square}
\end{equation}
where $\overline{\mathbf{z}}$ is the perfect $\mathsf{con}(\mathbf{z})$-square.
\end{lemma}

\begin{proof}
Without loss of generality, we may assume that $\mathbf{z}=(z_1 \cdots z_r)^2$.
Then
\begin{align*}
\mathbf{z} & \makebox[0.42in]{$=$}  z_1 \cdots z_iz_{i+1} \cdots z_r z_1 \cdots z_{i}z_{i+1} \cdots z_r  \\ %
& \makebox[0.42in]{$\stackrel{\eqref{basis S txaay}}{\approx}$}  z_1 \cdots z_{i+1}z_{i} \cdots z_r z_1 \cdots z_{i}z_{i+1} \cdots z_r  \\ %
& \makebox[0.42in]{$\stackrel{\eqref{basis S xaayt}}{\approx}$}  z_1 \cdots z_{i+1}z_{i} \cdots z_r z_1 \cdots z_{i+1}z_{i} \cdots z_r   \\ %
& \makebox[0.42in]{$=$} (z_1 \cdots z_{i+1}z_{i} \cdots z_r)^{2} .%
\end{align*}
Hence the identities ~$\circledS$ can be used to permute any letters within $\mathbf{z}$ in any manner. Specifically, the identities ~$\circledS$ can be used to permute any letters within $\mathbf{z}$ into alphabetical order, whence $\mathbf{z} \stackrel{\circledS}{\approx} \overline{\mathbf{z}}$.
\end{proof}

\begin{lemma}\label{lem: perfect square reduce}
Let $\mathbf{z}$ and $\mathbf{z}'$ be any squares with $\mathsf{con}(\mathbf{z}') \subseteq \mathsf{con}(\mathbf{z})$. Then the identities~$\circledS$ imply the identity
\begin{equation}
\mathbf{z}'\mathbf{z} \approx \mathbf{z} \approx \mathbf{z}\mathbf{z}'. \label{perfect square reduce}
\end{equation}
\end{lemma}

\begin{proof}
By symmetry$,$ it suffices to prove that $\mathbf{z}'\mathbf{z} \approx \mathbf{z}.$ Without loss of generality, we may assume that $\mathsf{t}(\mathbf{z}')=z$, and $\mathbf{z}=(z_1 \cdots z_i z z_{i+1} \cdots z_{r})^2$.
Then
\begin{align*}
z\mathbf{z} & \makebox[0.42in]{$=$} zz_1 \cdots z_iz z_{i+1} \cdots z_r z_1 \cdots z_iz z_{i+1} \cdots z_r  \\ %
& \makebox[0.42in]{$\stackrel{\eqref{basis S xx}}{\approx}$}  z(z_1 \cdots z_i)^2z z_{i+1} \cdots z_r z_1 \cdots z_iz z_{i+1} \cdots z_r   \\ %
& \makebox[0.42in]{$\stackrel{\eqref{basis S xy2x}}{\approx}$} (z_1 \cdots z_iz)^2 z_{i+1} \cdots z_r z_1 \cdots z_iz z_{i+1} \cdots z_r   \\ %
& \makebox[0.42in]{$\stackrel{\eqref{basis S xx}}{\approx}$}   z_1 \cdots z_iz z_{i+1} \cdots z_r z_1 \cdots z_iz z_{i+1} \cdots z_r   \\ %
& \makebox[0.42in]{$=$} \mathbf{z} .%
\end{align*}
It is easily seen how this procedure can be repeated so that the word $\mathbf{z}'\mathbf{z}$ can be  converted to $\mathbf{z}.$
\end{proof}

\begin{lemma}\label{lem: perfect square product}
Let $\mathbf{w}$ be any non-simple word such that $\mathsf{sim}(\mathbf{w}) = \emptyset$. Then the identities~$\circledS$ imply the identity $\mathbf{w} \approx \overline{\mathbf{w}},$ where
\[
\overline{\mathbf{w}}=\mathbf{z}_1 \cdots \mathbf{z}_{p_k}
\]
where
\begin{enumerate}
\item the words $\mathbf{z}_1^{(k)}, \ldots, \mathbf{z}_{p_k}^{(k)}\in \mathcal{X}^{+}$ are perfect squares;%
\item if $\cont(\mathbf{z}_\ell^{(k)}) \cap \cont(\mathbf{z}_g^{(k)}) \ne \emptyset$ for some $\ell<g$ and $\ell, g\in \{1, \ldots, p_k\}$, then $\cont(\mathbf{z}_\ell^{(k)}) \cap \cont(\mathbf{z}_g^{(k)}) \subseteq \cont(\mathbf{z}_j^{(k)})$ for each $\ell\leq j\leq g$;%
\item $\cont(\mathbf{z}_\ell^{(k)}) \nsubseteq \cont(\mathbf{z}_g^{(k)})$ for each $\ell \ne g$ and $\ell, g\in \{1, \ldots, p_k\}$; %
\end{enumerate}
\end{lemma}

\begin{proof}
Since each non-simple letter $x$ in $\mathbf{w}$ can be replaced by its square $x^2$ by applying the identities \eqref{basis S xx},
the word $\mathbf{w}$ can be written into the form of
\[
\mathbf{w}_k=\mathbf{z}_1 \cdots \mathbf{z}_p,
\]
where the words $\mathbf{z}_1, \ldots, \mathbf{z}_p \in \mathcal{X}^{+}$ are squares. Then by Lemma~\ref{lem: perfect square} we may assume that the words $\mathbf{z}_1, \ldots, \mathbf{z}_p$
are perfect squares. Hence the condition (CF2a) is satisfied.

If $z\in \cont(\mathbf{z}_\ell) \cap \cont(\mathbf{z}_g)$
for some $1 \leq \ell < g \leq p$, then
\begin{align*}
\mathbf{w} & \makebox[0.42in]{$=$}  \cdots \mathbf{z}_\ell \mathbf{z}_{\ell+1} \cdots \mathbf{z}_g \cdots  \\ %
& \makebox[0.42in]{$\stackrel{\eqref{perfect square reduce}}{\approx}$}  \cdots z^2\mathbf{z}_\ell \mathbf{z}_{\ell+1} \cdots \mathbf{z}_g z^2\cdots  \\ %
& \makebox[0.42in]{$\stackrel{\eqref{basis S xy2z2x}}{\approx}$}  \cdots z^2\mathbf{z}_\ell z \mathbf{z}_{\ell+1} z  \cdots z \mathbf{z}_g z^2 \cdots   \\ %
& \makebox[0.42in]{$\stackrel{\eqref{basis S xx}}{\approx}$}  \cdots z^2(z\mathbf{z}_\ell z)(z\mathbf{z}_{\ell+1} z)  \cdots (z \mathbf{z}_g z) z^2 \cdots  \\ %
& \makebox[0.42in]{$\stackrel{\eqref{perfect square}}{\approx}$} \cdots z^2 \overline{\mathbf{z}}_\ell \overline{\mathbf{z}}_{\ell+1} \cdots \overline{\mathbf{z}}_g z^2 \cdots \\ %
& \makebox[0.42in]{$\stackrel{\eqref{perfect square reduce}}{\approx}$} \cdots  \overline{\mathbf{z}}_\ell \overline{\mathbf{z}}_{\ell+1} \cdots \overline{\mathbf{z}}_g \cdots %
\end{align*}
where $\overline{\mathbf{z}}_j$ is the perfect $\mathsf{con}(\mathbf{z}_j)\cup \{z\}$-square for each $j= \ell, \ldots, g$. Now if $\cont(\mathbf{z}_\ell) \cap \cont(\mathbf{z}_g) = \emptyset$, then by repeating the above processes, each letter in $\cont(\mathbf{z}_\ell) \cap \cont(\mathbf{z}_g)$ can be put into $\mathbf{z}_j$ for each $\ell \leq j \leq g$. Hence we may assume that $\cont(\mathbf{z}_\ell) \cap \cont(\mathbf{z}_g) \subseteq \cont(\mathbf{z}_j)$, and so the condition (CF2b) is satisfied.

Suppose that $\cont(\mathbf{z}_\ell^{(k)}) \subseteq \cont(\mathbf{z}_g^{(k)})$ for some $\ell < g$.
Then $\cont(\mathbf{z}_\ell^{(k)}) \subseteq \cont(\mathbf{z}_{\ell+1}^{(k)})$ by the condition (CF2b). Hence by applying the identities \eqref{basis S xx} and \eqref{basis S xy2x}, the identity $\mathbf{z}_{\ell}^{(k)}\mathbf{z}_{\ell+1}^{(k)} \approx \mathbf{z}_{\ell+1}^{(k)}$ is hold, and so $\mathbf{z}_{\ell}^{(k)}$ can be deleted from $\mathbf{w}_k$.
Hence we may assume that the condition (CF2c) is satisfied.
\end{proof}

A word $\mathbf{w}$ is
said to be in \textit{canonical form} if
\begin{equation}
\mathbf{w} = \mathbf{w}_0 \prod_{i=1}^n (\mathbf{s}_i \mathbf{w}_i) \label{canon form}
\end{equation}
for some $n \geq 0$ such that the following conditions are all satisfied:
\begin{enumerate}
\item[(I)] the letters of $\mathbf{s}_1,\ldots,\mathbf{s}_n \in \mathcal{X}^+$ are simple in $\mathbf{w}$; %

\item[(II)] $\mathbf{w}_0, \mathbf{w}_n \in \mathcal{X}^{*}$ and $\mathbf{w}_1, \ldots, \mathbf{w}_{n-1} \in \mathcal{X}^{+}$ and for each $k=0, 1, \ldots, n$,
    \[
    \mathbf{w}_k=\mathbf{z}_1^{(k)} \cdots \mathbf{z}_{p_k}^{(k)}
    \]
    where
\begin{enumerate}
\item the words $\mathbf{z}_1^{(k)}, \ldots, \mathbf{z}_{p_k}^{(k)}\in \mathcal{X}^{+}$ are perfect squares;%
\item if $\cont(\mathbf{z}_\ell^{(k)}) \cap \cont(\mathbf{z}_g^{(k)}) \ne \emptyset$ for some $\ell<g$ and $\ell, g\in \{1, \ldots, p_k\}$, then $\cont(\mathbf{z}_\ell^{(k)}) \cap \cont(\mathbf{z}_g^{(k)}) \subseteq \cont(\mathbf{z}_j^{(k)})$ for each $\ell\leq j\leq g$;%
\item $\cont(\mathbf{z}_\ell^{(k)}) \nsubseteq \cont(\mathbf{z}_g^{(k)})$ for each $\ell \ne g$ and $\ell, g\in \{1, \ldots, p_k\}$; %
\item if $x\in \cont(\mathbf{z}_{\ell-1}^{(k)}) \setminus \cont(\mathbf{z}_\ell^{(k)})$ and $y\in \cont(\mathbf{z}_\ell^{(k)}) \setminus \cont(\mathbf{z}_{\ell-1}^{(k)})$ for some $k\in\{0, 1, \ldots, n\}$ and $\ell\in \{1,\ldots, p_k\}$, then $x, y$ satisfy neither of the following conditions:
\begin{enumerate}
\item $x \in \cont(\mathbf{w}_g)$ and $y \in \cont(\mathbf{w}_h)$ for some $g<h<k$ or $k<g<h$ or $h<k<g $;
\item $x, y \in \cont(\mathbf{w}_g)$ for some $g \ne k$.
\end{enumerate}
\end{enumerate}
\end{enumerate} %

An identity $\mathbf{u} \approx \mathbf{v}$ is \textit{canonical}
if the words $\mathbf{u}$ and $\mathbf{v}$ are in canonical form.

\begin{lemma} \label{canonical word}
Let $\mathbf{w}$ be any word\emph{.} Then there exists
some word $\overline{\mathbf{w}}$ in canonical form such that the
identities~$\circledS$ imply the identity $\mathbf{w}
\approx \overline{\mathbf{w}}.$
\end{lemma}

\begin{proof}
It suffices to convert the word $\mathbf{w}$, using the
identities~$\circledS$, into a word in canonical form.
It is easy to see that the word $\mathbf{w}$ can be written into the form of
\begin{equation}
\mathbf{w} = \mathbf{w}_0 \prod_{i=1}^n (\mathbf{s}_i \mathbf{w}_i)
\end{equation}
where $\mathbf{s}_1,\ldots, \mathbf{s}_n \in \mathcal{X}^{+}$, $\mathbf{w}_0, \mathbf{w}_n \in \mathcal{X}^{*}$,
$\mathbf{w}_1,\ldots, \mathbf{w}_{n-1} \in \mathcal{X}^{+}$, the letters of $\mathbf{s}_1,\ldots, \mathbf{s}_n$
are simple in $\mathbf{w}$ and the letters of $\mathbf{w}_0,\ldots, \mathbf{w}_{n}$ are non-simple in  $\mathbf{w}$. Hence the condition (CF1) is satisfied.

Since each non-simple letter $x_j$ in $\mathbf{w}_k$ can be replaced by its square $x_j^2$ by applying the identities \eqref{basis S xx},
the word $\mathbf{w}_k$ can be written into the form of
\[
\mathbf{w}_k=\mathbf{z}_1^{(k)} \cdots \mathbf{z}_{p_k}^{(k)},
\]
where the words $\mathbf{z}_1^{(k)}, \ldots, \mathbf{z}_{p_k}^{(k)} \in \mathcal{X}^{+}$ are squares.
The identities \eqref{basis S txaay} and \eqref{basis S xaayt} can be applied to alphabetically order the letters in each of
$\mathbf{z}_1^{(k)}, \ldots, \mathbf{z}_{p_k}^{(k)}$, and so we may assume that the words $\mathbf{z}_1^{(k)}, \ldots, \mathbf{z}_{p_k}^{(k)}$
are perfect squares. Hence the condition (CF2a) is satisfied.

If $z\in \cont(\mathbf{z}_\ell^{(k)}) \cap \cont(\mathbf{z}_g^{(k)})$
for some $\ell < g$ and $\ell, g\in \{1, \ldots, p_k\}$, then
\begin{align*}
\mathbf{w}_k & \makebox[0.42in]{$=$}  \cdots \mathbf{z}_\ell^{(k)} \mathbf{z}_{l+1}^{(k)} \cdots \mathbf{z}_g^{(k)} \cdots \\ %
& \makebox[0.42in]{$\stackrel{\eqref{basis S xy2x}}{\approx}$}  \cdots (z\mathbf{z}_\ell^{(k)}) \mathbf{z}_{\ell+1}^{(k)} \cdots (\mathbf{z}_g^{(k)} z)\cdots  \\ %
& \makebox[0.42in]{$\stackrel{\eqref{basis S xy2z2x}}{\approx}$}  \cdots z\mathbf{z}_\ell^{(k)} z \mathbf{z}_{\ell+1}^{(k)}  \cdots z \mathbf{z}_k^{(k)} z \cdots   \\ %
& \makebox[0.42in]{$\stackrel{\eqref{basis S xx}}{\approx}$}  \cdots (z\mathbf{z}_\ell^{(k)} z)(z\mathbf{z}_{\ell+1}^{(k)}z)  \cdots (z \mathbf{z}_g^{(k)} z) \cdots  \\ %
& \makebox[0.42in]{$\stackrel{\eqref{basis S xy2x}}{\approx}$} \cdots (z\mathbf{z}_\ell^{(k)})^2(z\mathbf{z}_{\ell+1}^{(k)})^{2} \cdots (z\mathbf{z}_g^{(k)})^2 \cdots .%
\end{align*}
Hence for each $\ell \leq j \leq g$, the perfect square $\mathbf{z}_j^{(k)}$ can be replaced by the perfect square $\overline{\mathbf{z}}_j^{(k)}$ where $\cont(\mathbf{z}_j^{(k)})\cup \{z\}=\cont(\overline{\mathbf{z}}_j^{(k)})$. Now if $\cont(\mathbf{z}_\ell^{(k)}) \cap \cont(\mathbf{z}_g^{(k)}) = \emptyset$, then by repeating the above processes, each letter in $\cont(\mathbf{z}_\ell^{(k)}) \cap \cont(\mathbf{z}_g^{(k)})$ can be put into $\mathbf{z}_j^{(k)}$ for each $\ell \leq j \leq g$. Hence we may assume that $\cont(\mathbf{z}_\ell^{(k)}) \cap \cont(\mathbf{z}_g^{(k)}) \subseteq \cont(\mathbf{z}_j^{(k)})$, and so the condition (CF2b) is satisfied.

Suppose that $\cont(\mathbf{z}_\ell^{(k)}) \subseteq \cont(\mathbf{z}_g^{(k)})$ for some $\ell < g$.
Then $\cont(\mathbf{z}_\ell^{(k)}) \subseteq \cont(\mathbf{z}_{\ell+1}^{(k)})$ by the condition (CF2b). Hence by applying the identities \eqref{basis S xx} and \eqref{basis S xy2x}, the identity $\mathbf{z}_{\ell}^{(k)}\mathbf{z}_{\ell+1}^{(k)} \approx \mathbf{z}_{\ell+1}^{(k)}$ is hold, and so $\mathbf{z}_{\ell}^{(k)}$ can be deleted from $\mathbf{w}_k$.
Hence we may assume that the condition (CF2c) is satisfied.

Let $x\in \cont(\mathbf{z}_{\ell-1}^{(k)}) \setminus \cont(\mathbf{z}_\ell^{(k)})$ and $y\in \cont(\mathbf{z}_\ell^{(k)}) \setminus \cont(\mathbf{z}_{\ell-1}^{(k)})$ for some $k\in\{0, 1, \ldots, n\}$ and $\ell\in \{1,\ldots, p_k\}$. Suppose that
$x, y$ satisfy one of the conditions (a), (b) and (c) in (CF3). Then since
\[
\cdots \mathbf{z}_{\ell-1}^{(k)} \mathbf{z}_{l}^{(k)} \cdots  \stackrel{\eqref{basis S xy2x}}{\approx} \cdots (\mathbf{z}_{\ell-1}^{(k)}x) (y\mathbf{z}_{l}^{(k)}) \cdots \stackrel{(\dag)}{\approx} \cdots \mathbf{z}_{\ell-1}^{(k)} (xy)^2\mathbf{z}_{l}^{(k)}) \cdots
\]
where $(\dag) = \eqref{basis S txsy}$ if $x, y$ satisfy the condition (a) or (b), and  $(\dag) = \{\eqref{basis S txaay}, \eqref{basis S xaayt}\}$ if $x, y$ satisfy the condition (c).  Hence in these cases, we may assume that there is a perfect square for the set $\{x, y\}$ between $\mathbf{z}_{\ell-1}^{(k)}$ and $\mathbf{z}_\ell^{(k)}$, and so the condition (CF3) is satisfied.
\end{proof}

%

\begin{lemma} \label{lem: non-id}
The variety $\mathsf{var} \{A^1 \times B^1\}$ does not satisfy the following identities
\begin{align}
xy^{2}tx & \approx  xy^{2}xtx, \label{id xyyxt} \\[0.04in] %
xty^{2}x & \approx xtxy^{2}x, \label{id txyyx} \\[0.04in] %
xsxtx  &\approx xstx, \label{id xtxsx}\\[0.04in] %
xxyy &\approx xy^2x \label{id xxyy}.
\end{align}
\end{lemma}

\begin{proof}
Let $x=e$, $y=d$ and $t=b$ in $A^1$. Then the left side of \eqref{id xyyxt} is $ed^2be =
a$, but the right side of \eqref{id xyyxt} is $ed^2ebe =0$, and so $A^1$ does not satisfy the identity \eqref{id xyyxt}. By a dual argument we may show that $B^{1}$ does not satisfy the identity \eqref{id txyyx}.

Let $x=e$, $s=c$ and $t=b$ in $A^1$. Then the left side of \eqref{id xtxsx} is $ecebe =
0$, but the right side of \eqref{id xtxsx} is $ecbe =0$, and so $A^1$ does not satisfy the identity \eqref{id xtxsx}.

Let $x=e$ and $y=d$ in $A^1$. Then the left side of \eqref{id xxyy} is $e^2d^2 =
c$, but the right side of \eqref{id xxyy} is $d^2e^2 =0$, and so $A^1$ does not satisfy the identity \eqref{id xxyy}.
\end{proof}

For any word $\mathbf{w}$, let $\mathsf{F_{SS}} (\mathbf{w})$ denote
the set of factors of $\mathbf{w}$ of length two that are formed by
simple letters:
\[
\mathsf{F_{SS}} (\mathbf{w}) = \{ xy \in \mathcal{X}^2 \mid \mathbf{w} \in \mathcal{X}^* xy \mathcal{X}^*, \, x,y \in \mathsf{sim} (\mathbf{w}) \}. %
\]
For example, if $\mathbf{w} = x^3 abcyxdy^2 efx$, then
$\mathsf{F_{SS}} (\mathbf{w}) = \{ ab, bc, ef \}$.

\begin{lemma} \label{sim}
Suppose that $\mathbf{w} \approx \mathbf{w}'$ is any identity
satisfied by the semigroup $S.$ Then
\begin{enumerate}
\item $\cont (\mathbf{w}) = \cont (\mathbf{w}')$ and $\simp (\mathbf{w}) = \simp (\mathbf{w}');$ %
\item $\mathsf{F_{SS}} (\mathbf{w}) = \mathsf{F_{SS}} (\mathbf{w}');$
\item $\mathbf{w}_{\simp} = \mathbf{w}_{\simp}'.$ %
\end{enumerate}
\end{lemma}

\begin{proof}
(1) and (3) follow from Lemma~\ref{prelim LEM J1 words} since the subsemigroup $\{0, 1, b, d\}$ of $A^1$ is isomorphic to $J^1$.

(2) follows from Lemma~1.10 of \cite{LeeZhang} since the variety $\mathsf{var} \{A^1 \times B^1\}$ does not satisfy the identity \eqref{id xtxsx}.
\end{proof}

For the remainder of this section, suppose that $\mathbf{w}
\approx \mathbf{w}'$ is any identity satisfied by the variety $\mathsf{var} \{A^1 \times B^1\}$, where the words
\begin{align}\label{w=w'}
\mathbf{w} = \mathbf{w}_0 \prod_{i=1}^n (\mathbf{s}_i \mathbf{w}_i) \quad \text{and} \quad \mathbf{w}' = \mathbf{w}_0' \prod_{i=1}^{n'} (\mathbf{s}_i' \mathbf{w}_i') %
\end{align}
are in canonical form. It follows from Lemma~\ref{sim} that $n=n'$ and $\mathbf{s}_k =\mathbf{s}_k'$  for each $k=0, \ldots, n$. The remainder of this section is devoted to the verification $\mathbf{w}_k = \mathbf{w}_k'$ for each $k=0, \ldots, n$.

\begin{lemma}\label{z con}
$\cont(\mathbf{w}_k)= \cont(\mathbf{w}_k')$.
\end{lemma}

\begin{proof}
Suppose that $x\in \cont(\mathbf{w}_k)\setminus \cont(\mathbf{w}_k')$. Since the variety $\mathsf{var} \{A^1 \times B^1\}$ satisfies the identity $\mathbf{w}
\approx \mathbf{w}'$, it is easy to see that the variety $\mathsf{var} \{A^1 \times B^1\}$ satisfies the identity $a\mathbf{w}b
\approx a\mathbf{w}'b$ where $a \ne b$ and $a, b \not\in \cont(\mathbf{w})=\cont(\mathbf{w}')$.
Then
\[
xsxtx \stackrel{\eqref{basis S xx}} \approx x(a\mathbf{w}b)_{\{s, t, x\}}x \approx x(a\mathbf{w}'b)_{\{s, t, x\}}x \stackrel{\eqref{basis S xx}} \approx xstx,
\]
where $s\in \cont(\mathbf{s}_k)$ if $k\geq 1$ and $s=a$ if $k=1$, and $t\in \cont(\mathbf{s}_{k+1})$ if $k< n$ and $t = b$ if $k=n$.
But this implies that the semigroup $S$ satisfies the identity \eqref{id xtxsx}, contradicting Lemma~\ref{lem: non-id}. Hence $\cont(\mathbf{w}_k)= \cont(\mathbf{w}_k')$.
\end{proof}

\begin{lemma}\label{z equivalent}
If for each $x,y \in \cont(\mathbf{w}_k)= \cont(\mathbf{w}_k')$, the condition
\begin{align}\label{equivalent}
x \prec_{\mathbf{w}_k} y \quad \mbox{if and only if}\quad x \prec_{\mathbf{w}_k'} y
\end{align}
is satisfied by the identity $\mathbf{w} \approx \mathbf{w}'$, then $\mathbf{w}_k=\mathbf{w}_k'$.
\end{lemma}

\begin{proof}
Without loss of generality, we may assume that
\[
\mathbf{w}_k= \mathbf{z}_1 \cdots \mathbf{z}_p \quad \mbox{and} \quad \mathbf{w}_k'= \mathbf{z}_1' \cdots \mathbf{z}_{p'}'.
\]
First we may show that
\begin{enumerate}
\item[(\dag)] if $\mathbf{z}_i$ is a perfect square factor of $\mathbf{w}_k$, then
$\mathbf{z}_i$ also is a perfect square factor of $\mathbf{w}_k'$.
\end{enumerate}

Suppose that $\cont(\mathbf{z}_i)$ is not any subset of $\cont(\mathbf{z}_g')$ for $g=1, \ldots, p'$. Since $\cont(\mathbf{w}_k)= \cont(\mathbf{w}_k')$ by Lemma~\ref{z con}, without loss of generality, we may assume that
$\cont(\mathbf{z}_l') \cup \cont(\mathbf{z}_g') \subseteq \cont(\mathbf{z}_i)$ for some $1 \leq \ell < g \leq p'$. By the condition (CF2c), let $x\in \cont(\mathbf{z}_l') \setminus \cont(\mathbf{z}_{l+1}')$ and $y\in \cont(\mathbf{z}_g') \setminus \cont(\mathbf{z}_l')$. It follows from (CF2b) that $x\not \in \cont(\mathbf{z}_{l+1}' \cdots \mathbf{z}_g'\cdots \mathbf{z}_p')$ and $y \not \in \cont(\mathbf{z}_1'\cdots \mathbf{z}_l')$, that is $x \prec_{\mathbf{w}_k'} y$. But $x \not\prec_{\mathbf{w}_k} y$ since $x,y \in \cont(\mathbf{z}_i)$, which contradicts the assumption. Therefore, we may assume that $\cont(\mathbf{z}_i) \subseteq \cont(\mathbf{z}_g')$ for some $g=1, \ldots, p'$.

Suppose that $\cont(\mathbf{z}_i) \subset \cont(\mathbf{z}_g')$. Without loss of generality, we may assume that $ z \in \cont(\mathbf{z}_g')\setminus \cont(\mathbf{z}_i)$. Then $z\in  \cont(\mathbf{z}_1 \cdots \mathbf{z}_{i-1}\mathbf{z}_{i+1} \cdots \mathbf{z}_p)$ by Lemma~\ref{z con} and $z\not\in  \cont(\mathbf{z}_1 \cdots \mathbf{z}_{i-1})\cap \cont(\mathbf{z}_{i+1} \cdots \mathbf{z}_p)$ by the condition (CF2b). Hence by symmetry, we may assume that $z\in  \cont(\mathbf{z}_1 \cdots \mathbf{z}_{i-1})\setminus \cont(\mathbf{z}_{i+1} \cdots \mathbf{z}_p)$, in particular, say $z\in  \cont(\mathbf{z}_\ell)\setminus \cont(\mathsf{z}_{\ell+1} \cdots \mathbf{z}_{i-1})$ for some $\ell < i$. Since  $z \not\prec_{\mathbf{w}_k'} x$ for each $x \in \cont(\mathbf{z}_i)\subseteq \cont(\mathbf{z}_g')$, it follows from the assumption that $x \in \cont(\mathbf{z}_1 \cdots \mathbf{z}_{\ell})$. Hence $\cont(\mathbf{z}_i) \subseteq \cont(\mathbf{z}_1 \cdots \mathbf{z}_{\ell})$. Now it follows from the condition (CF2b) that $\cont(\mathbf{z}_i) \subseteq \cont(\mathbf{z}_{\ell})$, which contradicts the condition (CF2c). Hence $\cont(\mathbf{z}_i) =\cont(\mathbf{z}_g')$. Now by the definition of perfect square it is easy to see that $\mathbf{z}_i= \mathbf{z}_g'$ and so (\dag) holds.

The converse of (\dag) also holds by symmetry. It then follows that $\mathbf{z}_i$ is a perfect square factor of $\mathbf{w}_k$ if and only if $\mathbf{z}_i$ also is a perfect square factor of $\mathbf{w}_k'$. Hence it follows from the conditions (CF2) that
$p=p'$ and $\{\mathbf{z}_1, \ldots, \mathbf{z}_{p}\} = \{\mathbf{z}_1', \ldots, \mathbf{z}_{p}'\}$.

Suppose that the occurrence of $\mathbf{z}_{i+1}$ precedes the occurrence of $\mathbf{z}_{i}$ in $\mathbf{w}_k'$.  By the condition (CF2c), let $x\in \cont(\mathbf{z}_i) \setminus \cont(\mathbf{z}_{i+1})$ and $y\in \cont(\mathbf{z}_{i+1}) \setminus \cont(\mathbf{z}_{i})$. Then $x \prec_{\mathbf{w}_k} y$, but $x \not\prec_{\mathbf{w}_k'} y$, which contradicts the assumption. Hence the order of occurrence of $\{\mathbf{z}_1, \ldots, \mathbf{z}_{p}\}$ in $\mathbf{w}_k$ is the same as the order of occurrence of $\{\mathbf{z}_1, \ldots, \mathbf{z}_{p}\}$ in $\mathbf{w}_k'$. Therefore $\mathbf{w}_k=\mathbf{w}_k'$.
\end{proof}

\begin{lemma} \label{zi=zi}
For each $x,y\in \cont(\mathbf{w}_k)=\cont(\mathbf{w}_k')$, if
$x \prec_{\mathbf{w}_k} y$, then $x \prec_{\mathbf{w}_k'} y$.
\end{lemma}

\begin{proof} Let
\[
\mathbf{w}_k= \mathbf{z}_1 \cdots \mathbf{z}_p.
\]
Seeking a contradiction, we may assume that
\begin{enumerate}
\item[(a)] $x\in \cont(\mathbf{z}_{\ell})\setminus \cont(\mathbf{z}_{\ell+1} \cdots \mathbf{z}_{p})$ and $y \in \cont(\mathbf{z}_{g})\setminus \cont(\mathbf{z}_{1} \cdots \mathbf{z}_{g-1})$ for some $1\leq \ell<g \leq p$;%
\item[(b)] for each $z\in \cont(\mathbf{z}_{\ell+1}\cdots \mathbf{z}_{g-1})$, if $x \prec_{\mathbf{w}_k} z$, then $x \prec_{\mathbf{w}_k'} z$;%
\item[(c)] for each $z\in \cont(\mathbf{z}_{\ell+1}\cdots \mathbf{z}_{g-1})$, if $z \prec_{\mathbf{w}_k} y$, then $z \prec_{\mathbf{w}_k'} y$;%
\item[(d)] $x \not\prec_{\mathbf{w}_k'} y$, that is, there exist some $x$ occur after some $y$ in $\mathbf{w}_k'$.
\end{enumerate}
There are three cases to consider.

\noindent{\bf Case~1.} $y\not\in \cont(\mathbf{w}_0 \cdots \mathbf{w}_{k-1})$ and $x \not\in \cont(\mathbf{w}_{k+1} \cdots \mathbf{w}_{n})$.
Then
\[
x^2y^2\stackrel{\eqref{basis S xx}} \approx x\mathbf{w}_{\{x, y\}} \approx x\mathbf{w}'_{\{x, y\}} \stackrel{\eqref{basis S xy2x}} \approx xy^2x,
\]
but this implies that the variety $\mathsf{var} \{A^1 \times B^1\}$ satisfies the identity \eqref{id xxyy}, contradicting Lemma~\ref{lem: non-id}.

\noindent{\bf Case~2.} $x \in \cont(\mathbf{w}_{k+1} \cdots \mathbf{w}_{n})$, say $x \in \cont(\mathbf{w}_{h}) \setminus \cont(\mathbf{w}_{k+1}\cdots \mathbf{w}_{h-1})$ for some $h>k$ and let $t$ be a simple letter in $\mathbf{s}_{h}$.

\noindent{\bf 2.1.} $g=l+1$. Then $y \not\in \cont (\mathbf{z}_{1}\cdots \mathbf{z}_{\ell})$ by (CF2b) and (a), and so by (CF3), $y \not\in \cont (\mathbf{w}_{0}\cdots \mathbf{w}_{k-1}\mathbf{w}_{h}\cdots \mathbf{w}_{n})$.  Hence it follows from Lemma~\ref{z con} that
\[
xy^2tx\stackrel{\eqref{basis S xx}} \approx x\mathbf{w}_{\{x, y, t\}} \approx x\mathbf{w}'_{\{x, y, t\}} \stackrel{\eqref{basis S xx}, \eqref{basis S xy2x}} \approx xy^2xtx,
\]
but this implies that the variety $\mathsf{var} \{A^1 \times B^1\}$ satisfies the identity \eqref{id xyyxt}, contradicting Lemma~\ref{lem: non-id}.

\noindent{\bf 2.2.} $g>l+1$. Then by (CF2c), there exist a letter $z$ such that $z \in \cont(\mathbf{z}_{\ell+1}) \setminus \cont(\mathbf{z}_{\ell})$ , and so $z \not\in \cont(\mathbf{z}_{1} \cdots \mathbf{z}_{\ell})$ by (CF2b). Suppose that $z\not\in \cont(\mathbf{z}_g)$. Then $z \not\in \cont(\mathbf{z}_{g} \cdots \mathbf{z}_{p})$ by (CF2b). Hence it is easy to see that $x \prec_{\mathbf{w}_{k}} z \prec_{\mathbf{w}_{k}} y$. Since $x \prec_{\mathbf{w}_{k}} z$, it follows from the assumption (b) that $x \prec_{\mathbf{w}_{k}'} z$. Since some $x$ occur after some $y$ in $\mathbf{w}_k'$ by (d),
\[
\mathbf{w}_k'= \cdots y\cdots x \cdots z\cdots,
\]
that is some $y$ occur before some $z$ in $\mathbf{w}_k'$. Hence $z \not\prec_{\mathbf{w}_{k}'} y$, which contradicts the assumption (c). Therefore $z\in \cont(\mathbf{z}_g)$. It follows from the condition (CF2c) that there exists a letter $s \ne z$ such that $s \in \cont(\mathbf{z}_{g-1}) \setminus \cont(\mathbf{z}_{g})$, and so $s\not\in \cont(\mathbf{z}_{g} \cdots \mathbf{z}_{p})$ by (CF2b). Hence it is easy to show that $x \prec_{\mathbf{w}_{k}} z$ and $s \prec_{\mathbf{w}_{k}} y$. It follows from the assumptions (b) and (c) that $x \prec_{\mathbf{w}_{k}'} z$ and $s \prec_{\mathbf{w}_{k}'} y$. Since some $x$ occur after some $y$ in $\mathbf{w}_k'$ by (d), it follows that $s \prec_{\mathbf{w}_{k}'} z$.

If $s\in \cont(\mathbf{w}_{k+1} \cdots \mathbf{w}_{h-1})$, say $s \in \cont(\mathbf{w}_{q}) \setminus \cont(\mathbf{w}_{k+1}\cdots \mathbf{w}_{q-1})$ for some $k<q<h$, then since
$s \in \cont(\mathbf{z}_{g-1}) \cap \cont(\mathbf{w}_{q})$ and $y\in \cont(\mathbf{z}_{g})\setminus \cont(\mathbf{z}_{1} \cdots \mathbf{z}_{g-1})$, it follows from (CF3) that $y \not\in \cont (\mathbf{w}_{0}\cdots \mathbf{w}_{k-1}\mathbf{w}_{q}\cdots\mathbf{w}_{h}\cdots \mathbf{w}_{n})$. Hence it follows from Lemma~\ref{z con} that
\[
xy^2tx\stackrel{\eqref{basis S xx}} \approx x\mathbf{w}_{\{x, y, t\}} \approx x\mathbf{w}'_{\{x, y, t\}} \stackrel{\eqref{basis S xx}, \eqref{basis S xy2x}} \approx xy^2xtx
\]
but this implies that the variety $\mathsf{var} \{A^1 \times B^1\}$ satisfies the identity \eqref{id xyyxt}, contradicting Lemma~\ref{lem: non-id}.

If $s\not\in \cont(\mathbf{w}_{k+1} \cdots \mathbf{w}_{h-1})$, then $s\not\in \cont(\mathbf{w}_{k+1}' \cdots \mathbf{w}_{h-1}')$ by Lemma~\ref{z con}.  Since $x\in \cont(\mathbf{z}_\ell) \cap \cont(\mathbf{w}_{h})$ and $z\in \cont(\mathbf{z}_{\ell+1}) \setminus \cont(\mathbf{z}_{1}\cdots \mathbf{z}_{\ell})$, it follows from (CF3) that $z \not\in \cont (\mathbf{w}_{0}\cdots \mathbf{w}_{k-1}\mathbf{w}_{h}\cdots \mathbf{w}_{n})$, and so $z \not\in \cont (\mathbf{w}_{0}'\cdots \mathbf{w}_{k-1}'\mathbf{w}_{h}'\cdots \mathbf{w}_{n}')$ by Lemma~\ref{z con}.
It follows that
\[
sz^2sts\stackrel{\eqref{basis S xx}} \approx \mathbf{w}_{\{z, t, s\}}s \approx\mathbf{w}'_{\{z, t, s\}}s \stackrel{\eqref{basis S xx}, \eqref{basis S xy2x}} \approx s^2z^2ts,
\]
but this implies that the variety $\mathsf{var} \{A^1 \times B^1\}$ satisfies the identity \eqref{id xyyxt}, contradicting Lemma~\ref{lem: non-id}.

\noindent{\bf Case~3.} $y\in \cont(\mathbf{w}_0 \cdots \mathbf{w}_{k-1})$. By arguments that are dual to Case 2 we may show that the variety $\mathsf{var} \{A^1 \times B^1\}$ satisfies either the identity \eqref{id xxyy} or the identity \eqref{id txyyx}, contradicting Lemma~\ref{lem: non-id}.
\end{proof}

\begin{proof}[\bf Proof of Theorem~\ref{S basis theorem}]
It is routine to verify that the identities $\circledS$ hold in semigroups $A^{1}$ and $A^{1}$
so that the variety $\mathsf{var} \{A^1 \times B^1\}$
satisfies the identities $\circledS$. It remains to show that any
identity $\mathbf{w} \approx \mathbf{w}'$ of the
variety $\mathsf{var} \{A^1 \times B^1\}$ is a consequence of the identities $\circledS$.
In the presence of Lemma~\ref{canonical word}, it suffices to
assume that the identity $\mathbf{w} \approx \mathbf{w}'$ is
canonical. Without loss of generality, we may assume that $\mathbf{w}$ and $\mathbf{w}'$ are in the form of \eqref{w=w'},
and $n=n'$ and $\mathbf{s}_k =\mathbf{s}_k'$ for each $k=0, \ldots, n$.
By Lemma~\ref{zi=zi} and its dual, it is easy to see that the condition \eqref{equivalent} is satisfied by the identity $\mathbf{w} \approx \mathbf{w}'$. Hence it follows from Lemma~\ref{z equivalent} that $\mathbf{w}_k = \mathbf{w}_k'$ for each $k=0, \ldots, n$.
Thus the identity $\mathbf{w} \approx \mathbf{w}'$ is
trivial and so is vacuously a consequence of the identities $\circledS$.
\end{proof}

\section{Monoid subvarieties of $\mathsf{var} \{A^1 \times B^1\}$}\label{sec lattice}

In this section, all monoid subvarieties of $\mathsf{var} \{A^1 \times B^1\}$ will be characterized and the monoid subvariety lattice of $\mathsf{var} \{A^1 \times B^1\}$ will be completely described. For convenience,
the monoid subvariety of $\mathsf{var} \{A^1 \times B^1\}$
defined by $\Pi$ is denoted by $\mathsf{var}_{\mathbb{M}} \{ \Pi \}$.

\begin{lemma}
\label{lem zi=zi}
Let $\mathbf{w} \approx \mathbf{w}'$ be any identity in canonical form where
\[
\mathbf{w} = \mathbf{w}_0 \prod_{i=1}^{k-1}(\mathbf{s}_i \mathbf{w}_i)\ \ \mathbf{s}_k  \underbrace{\mathbf{z}_{1} \cdots \mathbf{z}_{p}}_{\mathbf{w}_k}\prod_{i=k+1}^{n} (\mathbf{s}_i \mathbf{w}_i)  \quad \text{and} \quad \mathbf{w}' = \mathbf{w}_0' \prod_{i=1}^{n} (\mathbf{s}_i \mathbf{w}_i') %
\]
and $\cont(\mathbf{w}_i)= \cont(\mathbf{w}_i')$ for each $i=0, 1, \ldots, n$.
Suppose that $x$
and $y$ are non-simple letters of $\cont(\mathbf{w}_k) = \cont(\mathbf{w}_k')$ such that
\begin{enumerate}
\item[(a)] $x \prec_{\mathbf{w}_k} y$, say $x\in \cont(\mathbf{z}_{\ell})\setminus \cont(\mathbf{z}_{\ell+1} \cdots \mathbf{z}_{p})$ and $y \in \cont(\mathbf{z}_{g})\setminus \cont(\mathbf{z}_{1} \cdots \mathbf{z}_{g-1})$ for some $1\leq \ell<g \leq p$;%
\item[(b)] for each $z\in \cont(\mathbf{z}_{\ell+1}\cdots \mathbf{z}_{g-1})$, if $x \prec_{\mathbf{w}_k} z$, then $x \prec_{\mathbf{w}_k'} z$;%
\item[(c)] for each $z\in \cont(\mathbf{z}_{\ell+1}\cdots \mathbf{z}_{g-1})$, if $z \prec_{\mathbf{w}_k} y$, then $z \prec_{\mathbf{w}_k'} y$;%
\item[(d)] $x \not\prec_{\mathbf{w}_k'} y$.
\end{enumerate}
Then
\[
\mathsf{var}_{\mathbb{M}} \{ \mathbf{w} \approx \mathbf{w}' \} =
\mathsf{var}_{\mathbb{M}} \{\mathbf{w}^{*} \approx \mathbf{w}', \Lambda \}
\]
where $\mathbf{w}^{*}$ equal either
\[
\mathbf{w}_0 \prod_{i=1}^{k-1}(\mathbf{s}_i \mathbf{w}_i)\ \ \mathbf{s}_k  \underbrace{\mathbf{z}_{1} \cdots \mathbf{z}_{\ell}\cdots \mathbf{z}_{g-1}(xy)^2\mathbf{z}_{g}\cdots \mathbf{z}_{p}}_{\mathbf{w}_k}\prod_{i=k+1}^{n} (\mathbf{s}_i \mathbf{w}_i)
 \]
or
\[
\mathbf{w}_0 \prod_{i=1}^{k-1}(\mathbf{s}_i \mathbf{w}_i)\ \ \mathbf{s}_k  \underbrace{\mathbf{z}_{1} \cdots \mathbf{z}_{\ell}(xy)^2 \mathbf{z}_{\ell+1}\cdots \mathbf{z}_{g}\cdots \mathbf{z}_{p}}_{\mathbf{w}_k}\prod_{i=k+1}^{n} (\mathbf{s}_i \mathbf{w}_i)
 \]
and $\Lambda$ is some subset of $\{ (\ref{id xyyxt}), (\ref{id txyyx}), (\ref{id xtxsx})\}$.
\end{lemma}

\begin{proof}
There are three cases to consider.

\noindent{\bf Case~1.} $y\not\in \cont(\mathbf{w}_0 \cdots \mathbf{w}_{k-1})$ and $x \not\in \cont(\mathbf{w}_{k+1} \cdots \mathbf{w}_{n})$. Then by Case 1 of Lemma~\ref{zi=zi} that
\[
\mathbf{w} \approx \mathbf{w}' \Vdash \eqref{id xxyy}.
\]
Therefore
\begin{equation}
\label{z lemma display1}
\mathsf{var}_{\mathbb{M}} \{ \mathbf{w} \approx \mathbf{w}' \} = \mathsf{var}_{\mathbb{M}} \{ \eqref{id xxyy}, \mathbf{w} \approx \mathbf{w}' \}.%
\end{equation}
Now the deduction $\eqref{id xxyy} \Vdash \mathbf{w} \approx
\mathbf{w}^{*}$ holds since
\[
\begin{array}[c]{rcl}
\mathbf{w} \!\!\!\! & = & \!\!\!\! \mathbf{w}_0 \prod_{i=1}^{k-1} (\mathbf{s}_i \mathbf{w}_i)\ \ \mathbf{s}_k \mathbf{z}_{1} \cdots \mathbf{z}_{\ell} \mathbf{z}_{\ell+1} \cdots \mathbf{z}_{g-1}\mathbf{z}_{g} \cdots \mathbf{z}_{p} \prod_{i=k+1}^{n} \ (\mathbf{s}_i \mathbf{z}_i) \\[0.04in] %
& \stackrel{\eqref{basis S xy2x}}{\approx} & \!\!\!\! \mathbf{w}_0 \prod_{i=1}^{k-1} (\mathbf{s}_i \mathbf{w}_i)\ \ \mathbf{s}_k \mathbf{z}_{1} \cdots (\mathbf{z}_{\ell}x^2) \mathbf{z}_{\ell+1} \cdots \mathbf{z}_{g-1}(y^2\mathbf{z}_{g}) \cdots \mathbf{z}_{p} \prod_{i=k+1}^{n} \ (\mathbf{s}_i \mathbf{z}_i) \\[0.04in] %
& \stackrel{\eqref{id xxyy}}{\approx} & \!\!\!\! \mathbf{w}_0 \prod_{i=1}^{k-1} (\mathbf{s}_i \mathbf{w}_i)\ \ \mathbf{s}_k \mathbf{z}_{1} \cdots \mathbf{z}_{\ell}x \mathbf{z}_{\ell+1}x \cdots x \mathbf{z}_{g-1} x y^2 x\mathbf{z}_{g}  \cdots \mathbf{z}_{p} \prod_{i=k+1}^{n} \ (\mathbf{s}_i \mathbf{z}_i) \\[0.04in] %
& \stackrel{\eqref{basis S xy2x}}{\approx} & \!\!\!\! \mathbf{w}_0 \prod_{i=1}^{k-1} (\mathbf{s}_i \mathbf{w}_i)\ \ \mathbf{s}_k \mathbf{z}_{1} \cdots \mathbf{z}_{\ell}x \mathbf{z}_{\ell+1}x  \cdots x \mathbf{z}_{g-1} (xy)^2\mathbf{z}_{g} \cdots \mathbf{z}_{p} \prod_{i=k+1}^{n} \ (\mathbf{s}_i \mathbf{z}_i) \\[0.04in] %
& \stackrel{\eqref{id xxyy}}{\approx} & \!\!\!\! \mathbf{w}_0 \prod_{i=1}^{k-1} (\mathbf{s}_i \mathbf{w}_i)\ \ \mathbf{s}_k \mathbf{z}_{1} \cdots \mathbf{z}_{\ell}x^2 \mathbf{z}_{\ell+1} \cdots \mathbf{z}_{g-1} (xy)^2\mathbf{z}_{g} \cdots \mathbf{z}_{p} \prod_{i=k+1}^{n} \ (\mathbf{s}_i \mathbf{z}_i) \\[0.04in] %
& \stackrel{\eqref{basis S xy2x}}{\approx} & \!\!\!\! \mathbf{w}_0 \prod_{i=1}^{k-1} (\mathbf{s}_i \mathbf{w}_i)\ \ \mathbf{s}_k \mathbf{z}_{1} \cdots \mathbf{z}_{\ell} \mathbf{z}_{\ell+1} \cdots \mathbf{z}_{g-1} (xy)^2\mathbf{z}_{g} \cdots \mathbf{z}_{p} \prod_{i=k+1}^{n} \ (\mathbf{s}_i \mathbf{z}_i) \\[0.04in] %
& = & \!\!\!\! \mathbf{w}^{*}. %
\end{array}
\]
It follows from $(\ref{z lemma display1})$ that
$\mathsf{var}_{\mathbb{M}} \{ \mathbf{w} \approx \mathbf{w}' \} = \mathsf{var}_{\mathbb{M}} \{ (\ref{id xxyy}), \mathbf{w}^{*} \approx
\mathbf{w}'\}$.

\noindent{\bf Case~2.} $x \in \cont(\mathbf{w}_{k+1} \cdots \mathbf{w}_{n})$.
Then by Case 2 of Lemma~\ref{zi=zi} that
\[
\mathbf{w} \approx \mathbf{w}' \Vdash \eqref{id xyyxt}.
\]
Therefore
\begin{equation}
\label{z lemma display2}
\mathsf{var}_{\mathbb{M}} \{ \mathbf{w} \approx \mathbf{w}' \} = \mathsf{var}_{\mathbb{M}} \{ \eqref{id xyyxt}, \mathbf{w} \approx \mathbf{w}' \}.%
\end{equation}
Now the deduction $\eqref{id xyyxt} \Vdash \mathbf{w} \approx
\mathbf{w}^{*}$ holds since
\[
\begin{array}[c]{rcl}
\mathbf{w} \!\!\!\! & = & \!\!\!\! \mathbf{w}_0 \prod_{i=1}^{k-1} (\mathbf{s}_i \mathbf{w}_i)\ \ \mathbf{s}_k \mathbf{z}_{1} \cdots \mathbf{z}_{\ell} \mathbf{z}_{\ell+1} \cdots \mathbf{z}_{g-1}\mathbf{z}_{g} \cdots \mathbf{z}_{p} \prod_{i=k+1}^{n} \ (\mathbf{s}_i \mathbf{z}_i) \\[0.04in] %
& \stackrel{\eqref{basis S xy2x}}{\approx} & \!\!\!\! \mathbf{w}_0 \prod_{i=1}^{k-1} (\mathbf{s}_i \mathbf{w}_i)\ \ \mathbf{s}_k \mathbf{z}_{1} \cdots (\mathbf{z}_{\ell}x^2) \mathbf{z}_{\ell+1} \cdots \mathbf{z}_{g-1}(y^2\mathbf{z}_{g}) \cdots \mathbf{z}_{p} \prod_{i=k+1}^{n} \ (\mathbf{s}_i \mathbf{z}_i) \\[0.04in] %
& \stackrel{\eqref{id xxyy}}{\approx} & \!\!\!\! \mathbf{w}_0 \prod_{i=1}^{k-1} (\mathbf{s}_i \mathbf{w}_i)\ \ \mathbf{s}_k \mathbf{z}_{1} \cdots \mathbf{z}_{\ell}x \mathbf{z}_{\ell+1}x \cdots x \mathbf{z}_{g-1} x y^2 x\mathbf{z}_{g}  \cdots \mathbf{z}_{p} \prod_{i=k+1}^{n} \ (\mathbf{s}_i \mathbf{z}_i) \\[0.04in] %
& \stackrel{\eqref{basis S xy2x}}{\approx} & \!\!\!\! \mathbf{w}_0 \prod_{i=1}^{k-1} (\mathbf{s}_i \mathbf{w}_i)\ \ \mathbf{s}_k \mathbf{z}_{1} \cdots \mathbf{z}_{\ell}x \mathbf{z}_{\ell+1}x  \cdots x \mathbf{z}_{g-1} (xy)^2\mathbf{z}_{g} \cdots \mathbf{z}_{p} \prod_{i=k+1}^{n} \ (\mathbf{s}_i \mathbf{z}_i) \\[0.04in] %
& \stackrel{\eqref{id xxyy}}{\approx} & \!\!\!\! \mathbf{w}_0 \prod_{i=1}^{k-1} (\mathbf{s}_i \mathbf{w}_i)\ \ \mathbf{s}_k \mathbf{z}_{1} \cdots \mathbf{z}_{\ell}x^2 \mathbf{z}_{\ell+1} \cdots \mathbf{z}_{g-1} (xy)^2\mathbf{z}_{g} \cdots \mathbf{z}_{p} \prod_{i=k+1}^{n} \ (\mathbf{s}_i \mathbf{z}_i) \\[0.04in] %
& \stackrel{\eqref{basis S xy2x}}{\approx} & \!\!\!\! \mathbf{w}_0 \prod_{i=1}^{k-1} (\mathbf{s}_i \mathbf{w}_i)\ \ \mathbf{s}_k \mathbf{z}_{1} \cdots \mathbf{z}_{\ell} \mathbf{z}_{\ell+1} \cdots (xy)^2\mathbf{z}_{g} \cdots \mathbf{z}_{p} \prod_{i=k+1}^{n} \ (\mathbf{s}_i \mathbf{z}_i) \\[0.04in] %
& = & \!\!\!\! \mathbf{w}^{*}. %
\end{array}
\]
It follows from $(\ref{z lemma display2})$ that
$\mathsf{var}_{\mathbb{M}} \{ \mathbf{w} \approx \mathbf{w}' \} =
\mathsf{var}_{\mathbb{M}} \{(\ref{id xyyxt}), \mathbf{w}^{*} \approx
\mathbf{w}'\}$.

\noindent{\bf Case~3.} $y\in \cont(\mathbf{w}_0 \cdots \mathbf{w}_{k-1})$. Then by Case 3 of Lemma~\ref{zi=zi} that
\[
\mathbf{w} \approx \mathbf{w}' \Vdash \eqref{id txyyx}.
\]
Therefore
\begin{equation}
\label{z lemma display3}
\mathsf{var}_{\mathbb{M}} \{ \mathbf{w} \approx \mathbf{w}' \} = \mathsf{var}_{\mathbb{M}} \{ \eqref{id txyyx}, \mathbf{w} \approx \mathbf{w}' \}.%
\end{equation}
Now the deduction $\eqref{id txyyx} \Vdash \mathbf{w} \approx
\mathbf{w}^{*}$ holds since
\[
\begin{array}[c]{rcl}
\mathbf{w} \!\!\!\! & = & \!\!\!\! \mathbf{w}_0 \prod_{i=1}^{k-1} (\mathbf{s}_i \mathbf{w}_i)\ \ \mathbf{s}_k \mathbf{z}_{1} \cdots \mathbf{z}_{\ell} \mathbf{z}_{\ell+1} \cdots \mathbf{z}_{g-1}\mathbf{z}_{g} \cdots \mathbf{z}_{p} \prod_{i=k+1}^{n} \ (\mathbf{s}_i \mathbf{z}_i) \\[0.04in] %
& \stackrel{\eqref{basis S xy2x}}{\approx} & \!\!\!\! \mathbf{w}_0 \prod_{i=1}^{k-1} (\mathbf{s}_i \mathbf{w}_i)\ \ \mathbf{s}_k \mathbf{z}_{1} \cdots (\mathbf{z}_{\ell}x^2) \mathbf{z}_{\ell+1} \cdots \mathbf{z}_{g-1}(y^2\mathbf{z}_{g}) \cdots \mathbf{z}_{p} \prod_{i=k+1}^{n} \ (\mathbf{s}_i \mathbf{z}_i) \\[0.04in] %
& \stackrel{\eqref{id xxyy}}{\approx} & \!\!\!\! \mathbf{w}_0 \prod_{i=1}^{k-1} (\mathbf{s}_i \mathbf{w}_i)\ \ \mathbf{s}_k \mathbf{z}_{1} \cdots \mathbf{z}_{\ell} yx^2 y\mathbf{z}_{\ell+1}y \cdots y \mathbf{z}_{g-1} y \mathbf{z}_{g}  \cdots \mathbf{z}_{p} \prod_{i=k+1}^{n} \ (\mathbf{s}_i \mathbf{z}_i) \\[0.04in] %
& \stackrel{\eqref{basis S xy2x}}{\approx} & \!\!\!\! \mathbf{w}_0 \prod_{i=1}^{k-1} (\mathbf{s}_i \mathbf{w}_i)\ \ \mathbf{s}_k \mathbf{z}_{1} \cdots \mathbf{z}_{\ell} (xy)^2 \mathbf{z}_{\ell+1}y  \cdots y \mathbf{z}_{g-1} y \mathbf{z}_{g} \cdots \mathbf{z}_{p} \prod_{i=k+1}^{n} \ (\mathbf{s}_i \mathbf{z}_i) \\[0.04in] %
& \stackrel{\eqref{id xxyy}}{\approx} & \!\!\!\! \mathbf{w}_0 \prod_{i=1}^{k-1} (\mathbf{s}_i \mathbf{w}_i)\ \ \mathbf{s}_k \mathbf{z}_{1} \cdots \mathbf{z}_{\ell}(xy)^2 \mathbf{z}_{\ell+1} \cdots \mathbf{z}_{g-1} y^2\mathbf{z}_{g} \cdots \mathbf{z}_{p} \prod_{i=k+1}^{n} \ (\mathbf{s}_i \mathbf{z}_i) \\[0.04in] %
& \stackrel{\eqref{basis S xy2x}}{\approx} & \!\!\!\! \mathbf{w}_0 \prod_{i=1}^{k-1} (\mathbf{s}_i \mathbf{w}_i)\ \ \mathbf{s}_k \mathbf{z}_{1} \cdots \mathbf{z}_{\ell} (xy)^2 \mathbf{z}_{\ell+1} \cdots \mathbf{z}_{g-1}\mathbf{z}_{g} \cdots \mathbf{z}_{p} \prod_{i=k+1}^{n} \ (\mathbf{s}_i \mathbf{z}_i) \\[0.04in] %
& = & \!\!\!\! \mathbf{w}^{*}. %
\end{array}
\]
It follows from $(\ref{z lemma display3})$ that
$\mathsf{var}_{\mathbb{M}} \{ \mathbf{w} \approx \mathbf{w}' \} =
\mathsf{var}_{\mathbb{M}} \{ (\ref{id txyyx}), \mathbf{w}^{*} \approx
\mathbf{w}'\}$.
\end{proof}

\begin{lemma}
\label{*zi=zi}
Let $\mathbf{w} \approx \mathbf{w}'$ be any identity in canonical form where
\[
\mathbf{w} = \mathbf{w}_0 \prod_{i=1}^{n}(\mathbf{s}_i \mathbf{w}_i) \quad \text{and} \quad \mathbf{w}' = \mathbf{w}_0' \prod_{i=1}^{n} (\mathbf{s}_i \mathbf{w}_i') %
\]
and $\cont(\mathbf{w}_i)= \cont(\mathbf{w}_i')$ for each $i=0, 1, \ldots, n$.
Suppose that $\mathbf{w} \approx \mathbf{w}'$ does not satisfy the condition \eqref{equivalent}.
Then
\[
\mathsf{var}_{\mathbb{M}} \{ \mathbf{w} \approx \mathbf{w}' \} =
\mathsf{var}_{\mathbb{M}} \{\Lambda \}
\]
for some set $\Lambda$ of identities from $\{ (\ref{id xyyxt}), (\ref{id txyyx}), (\ref{id xtxsx})\}$.
\end{lemma}

\begin{proof}
Since  $\mathbf{w} \approx \mathbf{w}'$ does not satisfy the condition \eqref{equivalent}, there must exist letters $x,y\in \cont(\mathbf{w}_k)= \cont(\mathbf{w}_k')$ for some $k$ such that $x$ and $y$ satisfy the conditions (1)-(4) of Lemma~\ref{zi=zi} or its dual conditions. First if there exist some $x_1$ and $y_1$ satisfy the conditions (1)-(4) in $\mathbf{w}$. Then by Lemma~\ref{zi=zi},
the identities $\{ (\ref{id xyyxt}), (\ref{id txyyx}), (\ref{id xtxsx})\}$ can be used to convert $\mathbf{w} \approx \mathbf{w}'$ into $\mathbf{w}^{(1)} \approx \mathbf{w}'$ such that
\begin{enumerate}
\item $x_1 \prec_{\mathbf{w}_k} y_1$ and $x_1 \not \prec_{\mathbf{w}_k^{(1)}} y_1$;
\item if $x_1 \not\prec_{\mathbf{w}_k} z$ (resp. $z \not\prec_{\mathbf{w}_k} y_1$) for any $z \in \cont(\mathbf{w})$, then $x_1 \not\prec_{\mathbf{w}_k^{(1)}} z$ (resp. $z \not\prec_{\mathbf{w}_k^{(1)}} y_1$);
\item $z \not\prec_{\mathbf{w}_k} t$ if and only if $z \not\prec_{\mathbf{w}_k^{(1)}} t$ for each $z , t \in \cont(\mathbf{w}) \setminus \{x_1, y_1\}$.
\end{enumerate}
whence $\mathsf{var}_{\mathbb{M}} \{ \mathbf{w} \approx \mathbf{w}' \} =
\mathsf{var}_{\mathbb{M}}\{\mathbf{w}^{(1)} \approx \mathbf{w}', \Lambda^{(1)} \}$ for some set $\Lambda^{(1)}$ of identities from $\{ (\ref{id xyyxt}), (\ref{id txyyx}), (\ref{id xtxsx})\}$.

Now if there still exist letters $x_2, y_2$ in $\mathbf{w}^{(1)} \approx \mathbf{w}'$ such that $x_2, y_2$ does not satisfy the conditions (1)-(4) in $\mathbf{w}^{(1)}$, then the above procedure can be repeated to
construct an identity $\mathbf{w}^{(2)} \approx
\mathbf{w}'$ and some subset $\Lambda^{(2)} \subseteq \{ (\ref{id xyyxt}), (\ref{id txyyx}), (\ref{id xtxsx})\}$.
The construction of $\mathbf{w}^{(1)} \approx \mathbf{w}',
\mathbf{w}^{(2)} \approx \mathbf{w}', \ldots$ and
$\Lambda^{(1)}, \Lambda^{(2)}, \ldots$ cannot continue
indefinitely since that is bounded above by $C_{|\cont(\mathbf{w})|}^{2}$.

This
procedure can be repeated to obtain an identity $\mathbf{w}^{*}
\approx \mathbf{w}'$ that satisfies the property
\begin{align}\label{left equ}
\mbox{for any $k=0,1,\ldots, n$, $x,y \in \cont(\mathbf{w}_k')$, if $x \not\prec_{\mathbf{w}_k'} y$, then $x \not\prec_{\mathbf{w}_k^{*}} y$.}
\end{align}
Hence $\mathsf{var}_{\mathbb{M}} \{ \mathbf{w} \approx \mathbf{w}' \} =
\mathsf{var}_{\mathbb{M}}\{\mathbf{w}^{*} \approx \mathbf{w}', \Lambda \}$.

Now if there exist letters $x$ and $y$ satisfy the conditions (1)-(4) in $\mathbf{w}'$ of Lemma~\ref{zi=zi}, then by the same arguments to $\mathbf{w}'$, we can construct an identity $\mathbf{w}^{*} \approx \mathbf{w}^{'*}$ that satisfies the property
\begin{align}\label{right equ}
\mbox{for any $k=0,1,\ldots, n$, $x,y \in \cont(\mathbf{w}_k^{*})$, if $x \not\prec_{\mathbf{w}_k^{*}} y$, then $x \not\prec_{\mathbf{w}_k^{'*}} y$.} %
\end{align}
Hence $\mathbf{S} \{ \mathbf{w}^{*} \approx \mathbf{w}'\} =
\mathbf{S}\{\mathbf{w}^{*} \approx \mathbf{w}^{'*}, \Lambda \}$.
Since $\mathbf{w}^{*} \approx \mathbf{w}^{'*}$ satisfies the conditions \eqref{left equ} and \eqref{right equ}, by Lemma~\ref{z equivalent} that $\mathbf{w}^{*} \approx \mathbf{w}^{'*}$ is the trivial identity, and so $\mathsf{var}_{\mathbb{M}} \{ \mathbf{w}^{*} \approx \mathbf{w}'\} =
\mathsf{var}_{\mathbb{M}}\{\Lambda \}$.
\end{proof}

\begin{lemma}
\label{cont-w}
Let $\mathbf{w} \approx \mathbf{w}'$ be any identity in canonical form where
\[
\mathbf{w} = \mathbf{w}_0 \prod_{i=1}^{n}(\mathbf{s}_i \mathbf{w}_i)  \quad \text{and} \quad \mathbf{w}' = \mathbf{w}_0' \prod_{i=1}^{n} (\mathbf{s}_i \mathbf{w}_i'). %
\]
If $\cont(\mathbf{w}_i)= \cont(\mathbf{w}_i')$ for some $i=1, \ldots, n$, then
$\mathbf{w} \approx \mathbf{w}' \Vdash \eqref{id xtxsx}$.%
\end{lemma}

\begin{proof}
It follows from the proof of Lemma~\ref{z con}.
\end{proof}

\begin{lemma}
Let $\mathbf{w} \approx \mathbf{w}'$ be any identity.
\begin{enumerate}
\item If either $\cont (\mathbf{w}) \ne \cont (\mathbf{w}')$ or $\simp (\mathbf{w}) \ne \simp (\mathbf{w}')$, then $\mathbf{w} \approx \mathbf{w}' \Vdash x^2 \approx x \Vdash \eqref{id xtxsx}$; %
\item If $\mathsf{F_{SS}} (\mathbf{w}) \ne \mathsf{F_{SS}} (\mathbf{w}')$, then $\mathbf{w} \approx \mathbf{w}' \Vdash \eqref{id xtxsx}$; %
\item If $\mathbf{w}_{\simp} \ne  \mathbf{w}_{\simp}',$ then $\mathbf{w} \approx \mathbf{w}' \Vdash \eqref{id xtxsx}$.%
\end{enumerate}
\end{lemma}

\begin{proof}
(1) If $\mathbf{w} \approx \mathbf{w}'$ satisfy either $\cont (\mathbf{w}) \ne \cont (\mathbf{w}')$ or $\simp (\mathbf{w}) = \simp (\mathbf{w}')$, then monoid $N_2^1$ does not in the variety $\mathsf{var}_{\mathbb{M}}\{A^1 \times B^1\}$. It follows from Lemma~1.9 of \cite{LeeZhang}, that
\[
\mathbf{w} \approx \mathbf{w}' \Vdash y^2xy^2xy^2xy^2\approx y^2xy^2 \Vdash x^2 \approx x \Vdash \eqref{id xtxsx}
\]

(2) It follows from the proof of (2) of Lemma~\ref{sim}.

(3) If $\mathbf{w}_{\simp} \ne  \mathbf{w}_{\simp}',$ then we can derive some identity $\overline{\mathbf{w}} \approx \overline{\mathbf{w}'}$ from $\mathbf{w} \approx  \mathbf{w}'$ by deleting some letters in $\cont(\mathbf{w})\cup \cont(\mathbf{w}')$, such that
$\mathsf{F_{SS}} (\overline{\mathbf{w}}) \ne \mathsf{F_{SS}} (\overline{\mathbf{w}'})$. Then it follows from (2) that
$\mathbf{w} \approx \mathbf{w}' \Vdash \eqref{id xtxsx}$.
\end{proof}

\begin{theorem}
The subvariety lattice of $\mathsf{var}_{\mathbb{M}} \{A^1 \times B^1\}$ is as shown in Figure \ref{lattice}.
\end{theorem}

\begin{proof}
Let $\mathbf{w} \approx \mathbf{w}'$ be any nontrivial identity in canonical form where
\[
\mathbf{w} = \mathbf{w}_0 \prod_{i=1}^{n}(\mathbf{s}_i \mathbf{w}_i)  \quad \text{and} \quad \mathbf{w}' = \mathbf{w}_0' \prod_{i=1}^{n'} (\mathbf{s}_i' \mathbf{w}_i'). %
\]
If $\mathbf{w} \approx \mathbf{w}'$ does not satisfy
one of the following conditions:
\begin{enumerate}
\item[(a)] $\cont (\mathbf{w}) = \cont (\mathbf{w}')$ and $\simp (\mathbf{w}) = \simp (\mathbf{w}')$; %
\item[(b)] $\mathsf{F_{SS}} (\mathbf{w}) = \mathsf{F_{SS}}$; %
\item[(c)] $\mathbf{w}_{\simp} =  \mathbf{w}_{\simp}'$;
\item[(d)] $\cont(\mathbf{w}_i)= \cont(\mathbf{w}_i')$,%
\end{enumerate}
then $\mathbf{w} \approx \mathbf{w}' \Vdash \eqref{id xtxsx}$. Hence each subvariety that does not satisfy one of the conditions (a)-(d) is contained in the variety $\mathsf{var}_{\mathbb{M}}\circledS \cup \{\eqref{id xtxsx}\})$. It is easy to see that the variety $\mathsf{var}_{\mathbb{M}}\circledS \cup \{\eqref{id xtxsx}\}$ is just the variety $\mathsf{var}_{\mathbb{M}}\{A_0^1\}$ and its subvariety lattice can be found in \cite{Lee08}.

If $\mathbf{w} \approx \mathbf{w}'$ satisfy all of the conditions (a)-(d), then since $\mathbf{w} \approx \mathbf{w}'$ is nontrivial, it follows that $\mathbf{w}_k \ne \mathbf{w}_k'$ for some $k=0,\ldots,n$. Now it follows from Lemma~\ref{*zi=zi} that
\[
\mathsf{var}_{\mathbb{M}} \{ \mathbf{w} \approx \mathbf{w}' \} =
\mathsf{var}_{\mathbb{M}} \{\Lambda \}
\]
for some set $\Lambda$ of identities from $\{ (\ref{id xyyxt}), (\ref{id txyyx}), (\ref{id xtxsx})\}$.

It is easy to show that
\[
\eqref{id xxyy} \Vdash \eqref{id xyyxt},\quad \eqref{id xxyy} \Vdash \eqref{id txyyx},\quad \eqref{id xtxsx} \Vdash \eqref{id xyyxt},\quad \eqref{id xtxsx} \Vdash \eqref{id txyyx}.
\]
It follows from Proposition 4.3 of \cite{LeeLi11} that $\mathsf{var}_{\mathbb{M}}\{\circledS \cup \{\eqref{id xxyy}\}\}={\mathbb Q}^1$.
Clearly, the variety $\mathsf{var}_{\mathbb{M}} \circledS \cup \{\eqref{id xtxsx}\}=\mathsf{var}_{\mathbb{M}} \{A_0^1\}$ and $\mathsf{var}_{\mathbb{M}} \circledS \cup \{\eqref{id xxyy}\}=\mathsf{var}_{\mathbb{M}} \{Q^1\}$ are incomparable. It is routine to verify that $\mathsf{var}_{\mathbb{M}}\{A^{1}\}$ satisfy the identity $\eqref{id txyyx}$ but not $\eqref{id xyyxt}$, and $\mathsf{var}_{\mathbb{M}}\{A^{1})$ satisfy the identity $\eqref{id xyyxt}$ but not $\eqref{id txyyx}$, and the variety $\mathsf{var}_{\mathbb{M}}\circledS \cup \{\eqref{id xyyxt}\}$ and $\mathsf{var}_{\mathbb{M}}\circledS \cup \{\eqref{id txyyx}\}$ are incomparable. Hence $\mathsf{var}_{\mathbb{M}}\circledS \cup \{\eqref{id txyyx}\}=\mathsf{var}_{\mathbb{M}}\{A^{1}\}$ and $\mathsf{var}_{\mathbb{M}}\circledS \cup \{\eqref{id xyyxt}\}=\mathsf{var}_{\mathbb{M}}\{B^1\}$, which are the maximal proper monoid subvarieties of $\mathsf{var}_{\mathbb{M}}\{A^1 \times B^1\}$. Hence
the monoid subvariety lattice of $\mathsf{var}_{\mathbb{M}}\{A^1 \times B^1\}$ is as shown in Figure \ref{lattice}.
\end{proof}

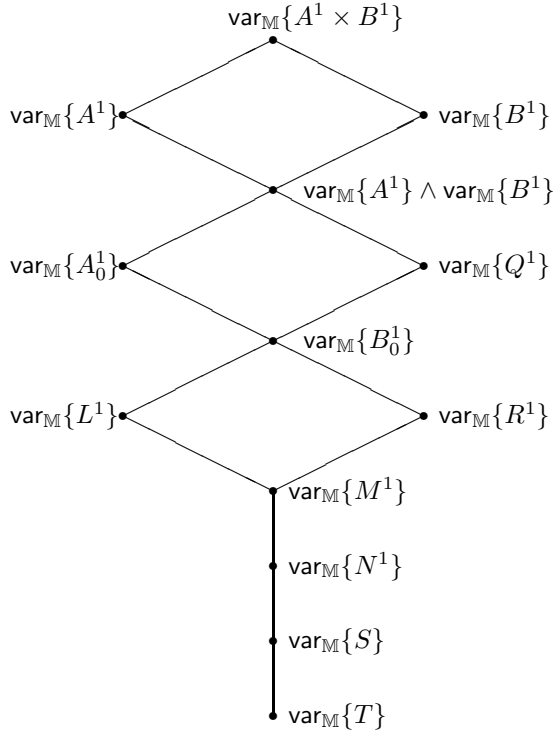
\begin{figure}[htb]
\begin{center}
\begin{picture}(200,250)(0,5)
\setlength{\unitlength}{1mm}
\put(10,40){\circle*{1}}\put(10,60){\circle*{1}}\put(10,80){\circle*{1}}%
\put(30,0){\circle*{1}}\put(30,10){\circle*{1}}\put(30,20){\circle*{1}}\put(30,30){\circle*{1}} \put(30,50){\circle*{1}}\put(30,70){\circle*{1}} \put(30,90){\circle*{1}}%
\put(50,40){\circle*{1}}\put(50,60){\circle*{1}}\put(50,80){\circle*{1}}%
\put(10,40){\line(2,1){40}}
\put(10,60){\line(2,1){40}}
\put(10,80){\line(2,1){20}}
\put(10,40){\line(2,-1){20}}%
\put(50,40){\line(-2,1){40}}
\put(50,60){\line(-2,1){40}}
\put(50,80){\line(-2,1){20}}
\put(30,30){\line(2,1){20}}%
\put(30,30){\line(0,-1){30}}
\put(32,0){\makebox(0,0)[l]{$\mathsf{var}_{\mathbb{M}}\{T\}$}}%
\put(32,10){\makebox(0,0)[l]{$\mathsf{var}_{\mathbb{M}}\{S\}$}}%
\put(32,20){\makebox(0,0)[l]{$\mathsf{var}_{\mathbb{M}}\{N^1\}$}}%
\put(32,30){\makebox(0,0)[l]{$\mathsf{var}_{\mathbb{M}}\{M^1\}$}}%
\put(-5,40){\makebox(0,0)[l]{$\mathsf{var}_{\mathbb{M}}\{L^1\}$}}%
\put(52,40){\makebox(0,0)[l]{$\mathsf{var}_{\mathbb{M}}\{R^1\}$}}%
\put(34,50){\makebox(0,0)[l]{$\mathsf{var}_{\mathbb{M}}\{B_0^1\}$}}%
\put(-5,60){\makebox(0,0)[l]{$\mathsf{var}_{\mathbb{M}}\{A_0^1\}$}}%
\put(52,60){\makebox(0,0)[l]{$\mathsf{var}_{\mathbb{M}}\{Q^1\}$}}%
\put(34,70){\makebox(0,0)[l]{$\mathsf{var}_{\mathbb{M}}\{A^1 \} \wedge \mathsf{var}_{\mathbb{M}}\{B^1 \}$}}%
\put(-5,80){\makebox(0,0)[l]{$\mathsf{var}_{\mathbb{M}}\{A^1 \}$}}%
\put(52,80){\makebox(0,0)[l]{$\mathsf{var}_{\mathbb{M}}\{B^1\}$}}%
\put(24,93){\makebox(0,0)[l]{$\mathsf{var}_{\mathbb{M}}\{A^1 \times B^1\}$}}%
\end{picture}
\end{center}
\caption{The monoid subvariety lattice $\mathsf{var}_{\mathbb{M}}\{A^1 \times B^1\}$} \label{lattice}
\end{figure}

{\bf Acknowledgments} The authors would like to express their
gratitude to Dr E. W. H. Lee for his comments and helps.

\end{document}